\documentclass{amsart}
\theoremstyle{plain}

\usepackage{bbm}

\usepackage[multiple]{footmisc}
\usepackage{amsthm}
\usepackage{amssymb}
\usepackage{amsmath}
\usepackage{xfrac}
\usepackage{xparse}
\usepackage{faktor}
\usepackage{array}
\usepackage{mathtools}
\usepackage{stmaryrd}
\usepackage{inputenc}
\usepackage{csquotes}
\usepackage{nicefrac}
\usepackage{tikz-cd}

\newtheorem{prop}{Proposition}	
	
\newtheorem{definition}[prop]{Definition}
\newtheorem{theorem}[prop]{Theorem}

\newtheorem{cor}[prop]{Corollary}
\newtheorem{lemma}[prop]{Lemma}
\newtheorem{remark}[prop]{Remark}

\numberwithin{prop}{section}
\numberwithin{equation}{section}

\newtheorem*{definition*}{Definition}
\newtheorem*{theorem*}{Theorem}
\newtheorem*{remark*}{Remark}

\newcommand{\NN}{\ensuremath{\mathbb{N}}}
\newcommand{\ZZ}{\ensuremath{\mathbb{Z}}}

\newcommand{\IRS}[1]{\ensuremath{\mathrm{IRS}\left( #1 \right) }}

\newcommand{\IRSfi}[1]{\ensuremath{\mathrm{IRS}_{\mathrm{fi}}\left( #1 \right) }}

\newcommand{\nrm}{\ensuremath{\vartriangleleft }}

\newcommand{\Sub}[1]{\ensuremath{\mathrm{Sub}\left({#1}\right)}}

\newcommand{\Subfi}[1]{\ensuremath{\mathrm{Sub}_{\mathrm{fi}}\left({#1}\right)}}
\newcommand{\Pow}[1]{\ensuremath{\mathrm{Pow}\left({#1}\right)}}
\newcommand{\Probs}[1]{\ensuremath{\mathcal{M}\left({#1}\right)}}

\newcommand{\setr}[1]{\ensuremath{\llbracket #1 \rrbracket}}
\newcommand{\U}{\ensuremath{U}}
\newcommand{\UU}[1]{\ensuremath{\U_{#1}}}
\newcommand{\UP}[1]{\ensuremath{A_{#1}(\mathcal{P})}}

\title[Uncountably many permutation stable groups]{Uncountably many permutation stable groups}
\author{Arie Levit \and Alexander Lubotzky}

\begin{document}

\maketitle

\begin{abstract} 
In a 1937 paper B.H. Neumann constructed an uncountable family of $2$-generated groups.  We prove that all of his groups are permutation stable by analyzing the structure of their invariant random subgroups.
\end{abstract}

\section{Introduction}
\label{sec:intro}

 Let $\mathrm{S}(n)$ denote the symmetric group of degree $n \in \NN$ with   the bi-invariant Hamming metric $d_n$  given by
$$ d_n(\sigma, \tau) = 1 - \frac{1}{n} \left|\mathrm{Fix}(\sigma^{-1} \tau)\right|.$$

Let $G$ be a finitely generated group. An \emph{almost-homomorphism} of $G$ is a sequence of set theoretic maps $f_n : G \to \mathrm{S}(n )$   satisfying
$$ d_{n } \left(f_n(g) f_n(h), \, f_n(gh)\right) \xrightarrow{n\to\infty} 0, \quad \forall g,h \in G.$$
The   almost-homomorphism $f_n$ is \emph{close to a homomorphism} if there is a sequence of group homomorphisms $\rho_n : G \to \mathrm{S}(n )$ satisfying
$$ d_{n } \left(\rho_n(g), \, f_n(g)\right) \xrightarrow{n\to\infty} 0 \quad \forall g\in G.$$

\begin{definition}
\label{def:stable in permutations}
The group $G$ is \emph{permutation stable} (or \emph{P-stable} for short) if every almost-homomorphism of $G$ is close to a homomorphism.
% every $\varepsilon > 0$ there is some $\delta = \delta(\varepsilon) > 0$ and some \emph{finite} subset of defining relations $T \subset R$ with the following property --- for every $n \in \NN$ and subset $\{\pi_\sigma\}_{\sigma \in \Sigma} \subset \mathrm{S}(n)$  satisfying 
%$$ d_n(r(\pi_\sigma), \, \mathrm{id} ) < \delta \quad \forall r \in T$$
%there is a group homomorphism $\rho : G \to \mathrm{S}(n)$ satisfying
%$$ d_n(\rho(\sigma),\, \pi_\sigma) < \varepsilon  \quad \forall \sigma \in \Sigma.$$
\end{definition}

Until quite recently only   few groups were known to be permutation stable: free groups (trivially), finite groups \cite{glebsky2009almost} and abelian groups \cite{arzhantseva2015almost}.  The situation was changed with \cite{becker2019stability}. That work established a connection between stability and invariant random subgroups  of a given amenable group $G$. 
 
Recall that an invariant random subgroup $\mu$  of $G$ is  a conjugation invariant   probability measure on the   space of  all   subgroups of $G$.  The invariant random subgroup $\mu$ is  \emph{co-sofic} if $\mu$  is a limit   of invariant random subgroups supported on finite index subgroups.

\begin{theorem}[\cite{becker2019stability}]
\label{thm:BLT}
A finitely generated  amenable group  $G$ is permutation stable   if and only if every invariant random subgroup of $G$ is co-sofic.
\end{theorem}

 This new viewpoint enabled the authors of \cite{becker2019stability} to deduce that some groups whose invariant random subgroups are easy to analyze are  permutation stable, namely polycyclic-by-finite  as well as the  Baumslag--Solitar groups $B(1,n)$ for all $ n \in \ZZ$. That work motivated a deeper analysis of the invariant random subgroups of more complicated groups: the lamplighter group, as well as any wreath product of finitely generated abelian groups, is shown in \cite{levit2019infinitely} to be  permutation stable, while \cite{zheng2019rigid} showed that the Grigorchuk group is such.
 The goal of this paper is to go further and   prove

\begin{theorem}
\label{thm:main theorem }
There exist uncountably many   $2$-generated permutation stable groups.
\end{theorem}

  B.H. Neumann \cite{neumann1937some} gave an explicit construction of an uncountable family of pairwise non-isomorphic $2$-generator groups. We   actually prove that all of those groups are permutation stable.

The B.H. Neumann family of    groups is constructed as follows.
Let $\mathcal{P} = \left(n_i\right)_{i \in \NN}$ be any monotone increasing sequence of odd integers with $n_1 \ge 5$. For all $ i\in \NN$ write $n_i = 2r_i + 1$ and denote $\setr{r_i} = \{ n \in \ZZ \: : \: |n| \le r_i\}$ so that $|\setr{r_i}| = n_i$. Let 
$$A(\mathcal{P}) = \prod_{i\in\NN} \mathrm{Alt}(\setr{r_i})$$  where $ \mathrm{Alt}(\setr{r_i})$ is the group of all even permutations of the set $\setr{r_i}$ for each $i \in \NN$. The group   $A(\mathcal{P})$ is profinite.

The B.H. Neumann group $G(\mathcal{P})$ associated to the sequence $\mathcal{P}$ is the countable subgroup of $A(\mathcal{P}) $ generated by the two elements $\tau = (\tau_i)_{i\in\NN}$ and $\sigma = (\sigma_i)_{i\in\NN}$,  where each permutation $\tau_i$ is the $3$-cycle 
$$ \tau_i = \left(-1,0,1\right) \in \mathrm{Alt}(\setr{r_i})$$
and each permutation $\sigma_i$ is the full $n_i$-cycle
$$ \sigma_i = \left(-r_i,-r_i+1,\ldots,-1,0,1,\ldots,r_i-1,r_i\right) \in \mathrm{Alt}(\setr{r_i}).$$

%This procedure gives a group $G(\mathcal{P})$ for every such sequence $\mathcal{P}$.
 B.H. Neumann showed  --- see also \S\ref{sec:Neumann groups} --- that the groups $G(\mathcal{P})$ and $G(\mathcal{P'})$ are isomorphic if and only if the two sequences $\mathcal{P}$ and $ \mathcal{P}'$ as above coincide. So his family of groups is indeed uncountable. It has already been observed in \cite{lubotzky1993groups} --- see also \S\ref{sec:Neumann groups}   --- that all these  groups are amenable.  So their stability, which is a purely group theoretic property, can be studied via their invariant random subgroups by Theorem \ref{thm:BLT}. We will prove
 
 \begin{theorem}
 \label{thm:main result IRS}
 All invariant random subgroups of each B.H. Neumann group $G(\mathcal{P})$ are co-sofic.
 \end{theorem}

In this work, as in \cite{levit2019infinitely}, the proof of our Theorem \ref{thm:main result IRS}  uses methods of ergodic theory, and relies on the Lindenstrauss pointwise erogdic theorem for amenable groups \cite{lindenstrauss2001pointwise}.

\subsection*{Outline of the paper}
 The paper is organized as follows.
 
  In \S\ref{sec:Neumann groups} we recall in a slightly modified way the  construction of the B.H. Neumann groups $G = G(\mathcal{P})$. We analyze their structure and provide two short exact sequences to be used later. For instance,   we will show that the derived subgroup $G'$  of $G$ is mapped onto $\mathrm{Alt}_\text{fin}(\ZZ)$, the group of all even permutations of the set $\ZZ$ with finite support.
  
   In \S\ref{sec:Weiss approx} we will recall some basic definitions and facts about invariant random subgroups, as well as the notion of Weiss approximations relying on the pointwise ergodic theorem.
   
   In \S\ref{ref:vershik} we recall Vershik's theorem which gives a complete classification of the invariant random subgroups of the group $\mathrm{Sym}_\text{fin}(\ZZ)$. This classification will be used in the proof of the main theorem. 
   
   The proof will be given in \S\ref{sec:external case}. One main ingredient of the proof is the explicit construction done in \S\ref{sec:internal case} of a Weiss approximation to any subgroup of $G$ contained in the derived subgroup $G'$.  The other main ingredient has to do with showing that every invariant random subgroup of $G$ not entirely contained in the derived subgroup $G'$ is already supported on finite index subgroups. 
     
An interesting aspect of the proof is that $G$ admits a certain normal subgroup $\U$ so that $G/ \U$ is not residually finite. This means that the atomic invariant random subgroup of $G$ supported on the this normal subgroup $\U$ cannot be approximated \enquote{from above} by invariant random subgroups corresponding to   finite index subgroups containing $\U$. Instead, it can  only be approximated \enquote{from the side}, so to speak, using invariant random subgroups which do not come from $G/\U$, see \S\ref{sec:internal case}.

\vspace{5pt}
\emph{We follow the convention $x^y = y^{-1} x y$ for conjugation.}

\subsection*{Acknowledgements}
%The authors would like to thank Oren Becker for stimulating discussions and in particular for bringing to our attention Corollary \ref{cor:residually finite amenable has Foler sequence of finite to one transverals}. 

The first named author would like to thank Mikl\'os Ab\'ert for insightful conversations on invariant random subgroups. The first named author is indebted for support from the NSF. The second named author   is indebted for support from the NSF, the ISF and the European
Research Council (ERC) under the European Union's Horizon 2020
research and innovation program (grant agreement No. 692854).
%\subsection*{Outline of the paper}
%EXPLAIN what is contained in every section.
%\subsection*{The split and the non-split case}
%	TODO EXPLAIN ABOUT SPLIT AND NNOAssume that the group $G$ is split so that $G \cong Q \ltimes N$. In this case we may identify $\widehat{Q}$ with a subgroup of $G$ and there is an obvious choice of a lift. This lift is obviously tame in our terminology. In particular, assuming that $G$  split would make a large part of the work in this section considerably simpler.
%	
%MORE THINGS WE DON'T NEED
%\begin{enumerate}
%\item the partition $P_i = P'_i + P''_i$,
%\end{enumerate}
%
% 

%\marginpar{VERIFY ALL CONJUGATIONS THROUGHOUT WORK CHANGE TO NEW NOTATION}

\section{The B.H.  Neumann Groups}
\label{sec:Neumann groups}

We recall in detail the definition of the B.H. Neumann groups and present some of their properties to   be used later.

\subsection*{Basic notations} 
For each $r \in \NN$ write
$$ \setr{r} = \{ m \in \ZZ \: : \: |m| \le r \}$$
so that $\setr{r}$ is a set of integers of size $ 2r+1$. Let   $\UU{r}$ denote the group of all even permutations of the set $\setr{r}$. So $\UU{r}$ is naturally isomorphic to the alternating group $A_{2r+1}$. Let $\UU{\infty}$ denote the group $ \mathrm{Alt}_\text{fin}(\ZZ)$ of all even  permutations of the set $\ZZ$ with finite support. Note that $\UU{\infty} = \varinjlim_r \UU{r}$  in a natural way. Therefore $\UU{\infty}$ is a simple group.

Let $\mathcal{P} = (n_i)_{i\in\NN}$ be a monotone sequence of odd integers satisfying $n_1 \ge 5$.  For each $i \in \NN$ denote $n_i = 2r_i + 1$ and $\UP{i} = \UU{r_i}$. Consider the profinite group 
$$A(\mathcal{P}) = \prod_{i \in \NN} \UP{i}$$ 

Let $G = G(\mathcal{P})$ be the countable subgroup of $A(\mathcal{P}) $ generated by the two elements $\tau = (\tau_i)_{i\in\NN}$ and $\sigma = (\sigma_i)_{i\in\NN}$, where   for each $i \in \NN$ we have
$$ \tau_i = \left(-1,0,1\right) \in \UP{i}$$
and 
$$ \sigma_i = \left(-r_i,-r_i+1,\ldots,-1,0,1,\ldots,r_i-1,r_i\right) \in \UP{i}.$$
The group $G(\mathcal{P})$ is the \emph{B.H. Neumann group} associated to the sequence $\mathcal{P}$.

\subsection*{Sets of generators}  For each $r \in \NN$  the following collection of $3$-cycles
\begin{align*}
C_r 
  =\{ (-r, -r+1,-r +2), \ldots, (-1,0,1), \ldots, (r -2, r-1,r)  \}
 \end{align*}
 of size $2r- 1$ is well known to   generate  the   finite alternating group $\UU{r}$.   For each $ i \in \NN$ consider the following generating set  
$$C_i(\mathcal{P}) = C_{r_i} = \{ \tau_i^{\sigma_i^j} \: : \: j \in \setr{r_i - 1} \}$$   for the group $\UP{i}$. 
 In particular $\UP{i}$ is generated by the two elements $\tau_i$ and $\sigma_i$ so that $G(\mathcal{P})$ surjects onto each coordinate of $A(\mathcal{P})$. 
 The following standard fact of elementary group theory  implies that   $G(\mathcal{P})$ is a  dense subgroup of  the profinite group  $A(\mathcal{P})$.

\begin{lemma}
\label{lem:group theory}
Let $\Gamma_1, \ldots, \Gamma_l$ be a family of pairwise non-isomorphic simple groups. If $H$ is a subgroup of the direct product $\Gamma = \prod_{i=1}^l \Gamma_i$ surjecting onto each coordinate  then $H = \Gamma$. 
\end{lemma}

\subsection*{Local finiteness} For each $ i \in\NN$, let  $L_i(\mathcal{P})$ denote the subgroup of $G(\mathcal{P})$ given by
$$ L_i(\mathcal{P}) = \left<  \tau^{\sigma^{j}} \: : \: j \in \setr{r_i-1}  \right>.$$
 The projection of $L_i(\mathcal{P})$ to any coordinate $\UP{j}$ with $ j \le i$  is surjective  for it contains the generating set $C_j(\mathcal{P})$, while if $j > i$ the  image of the  projection to $\UP{j}$ is isomorphic to $\UP{i}$ in a natural way.  In any case, the support of every such projection is contained inside $\setr{r_i}$.

In fact, for each $ i \in \NN$ the projection of the subgroup $L_i(\mathcal{P})$ onto $\prod_{j=1}^i \UP{j}$ is  an isomorphism. It is surjective  according to Lemma \ref{lem:group theory}. To see that this projection is injective, note that every element $g = (g_k)_{k\in\NN}$ belonging to $L_i(\mathcal{P})$ satisfies $g_j = g_i$ for all $ j \ge i$. It follows that  each subgroup  $L_i(\mathcal{P})$ is finite.

The    subgroup $N  = N(\mathcal{P})$ of $G(\mathcal{P})$ given by
$$N(\mathcal{P})   = \bigcup_{i\in \NN} L_i(\mathcal{P}).$$ 
 is locally finite. It is also the normal closure of the element $\tau$ in the group $G(\mathcal{P})$.  Therefore $G/N$ is a cyclic group   generated by the image of   $\sigma$. This is an infinite cyclic group since $\sigma$ is of infinite order and hence no non-trivial power of it can belong to  $N$. Moreover $G$ is the semi-direct product of $N$ and of the cyclic group generated by $\sigma$.

It is easy to detect the elements of $G(\mathcal{P})$ belonging to the subgroup   $N(\mathcal{P})$. An element $g = (g_k)_{k \in \NN} \in G(\mathcal{P})$ belongs to $N(\mathcal{P})$ if and only if there is some $ i \in \NN$ such that the support of $g_k$ is contained in $\setr{i}$ for all $ k \in \NN$. For every element $g \in N$, let $i = i(g) \in \NN$ be the smallest number such that $g \in L_i(\mathcal{P})$. The element $g$ acts \emph{diagonally}  in all coordinates $\UP{j}$ with $ j \ge i(g)$, namely $g_j = g_{i(g)}$ for all such $j$, provided that we identify $\UP{i(g)}$ with a subgroup of $\UP{j}$ in the natural way.
\subsection*{The tail map}

The \emph{tail} (or the \emph{diagonal part}) of the element $g = (g_k)_{k\in\NN} \in N $ is 
$$ t(g) = g_{i(g)} \in \UP{i(g)} = \UU{r_{i(g)}} \subset \UU{\infty}.$$
We consider the tail $t(g)$  as an even permutation of the set $\ZZ$ with finite support by naturally regarding the set $\setr{r_i}$ as a subset of $\ZZ$. It is easy to see that the map $t : N \to \UU{\infty}$ is a homomorphism. Let us summarize.

\begin{prop}
\label{prop:properties of G(P)}
The  group $G = G(\mathcal{P})$ satisfies 
\begin{enumerate}
	\item $G$ is a residually finite group generated by the two elements $\sigma$ and $\tau$,
	\item the normal closure $N = N(\mathcal{P})$ of the element $\tau$ in $G$ is locally finite and   $G =    N \rtimes \left<\sigma\right>$,
	\item $G$ is locally finite-by-cyclic and hence amenable, and
	\item there is a surjective tail homomorphism $t : N \to \UU{\infty}$.
\end{enumerate}
\end{prop}
\begin{proof}
Everything was already noticed before, except for the fact that $t$ is surjective. To see this, consider some element $h \in \UU{\infty}$. The support of $h$ is contained in $\setr{r_i}$ for some $i \in \NN$ sufficiently large. Then $h = t(g)$ for any element $g \in L_i(\mathcal{P})$ whose projection on the $\UP{i}$ coordinate coincides with $h$.
\end{proof}

We would like to  extend the tail map $t$ to the entire group $G(\mathcal{P})$. To this end it is desirable to replace the range   of   $t$ by a larger group.
Let $\bar{\sigma}$ be the permutation of the set $\ZZ$ given by $\bar{\sigma}(x) = x+ 1$ for all integers $x \in \ZZ$. Let $V$ denote the following subgroup of $\mathrm{Sym}(\ZZ)$
$$ V = \UU{\infty} \rtimes \left<\bar{\sigma}\right>. $$

\begin{prop}
\label{prop:tail map extends}
Consider the group $G = G(\mathcal{P})$.
\begin{enumerate}
	\item 
		\label{it:tail extends}
		The tail map $t$ naturally extends to a surjective   homomorphism 
	$ \widetilde{t} : G \to V$.

	\item Let $ \U(\mathcal{P})$ denote the kernel $\ker \widetilde{t} = \ker t$, i.e the subgroup of $N$ consisting of all elements with trivial tail. Then $\U(\mathcal{P}) = \oplus_{i\in\NN} \UP{i}$.
	\label{it:kernel of tail}
\end{enumerate}
\end{prop}
\begin{proof}
 Recall that $G = N \rtimes \left<\sigma\right>$. It clear from the definitions  that the map $\widetilde{t} : G \to V$ given by $\widetilde{t}(\sigma^k) = \bar{\sigma}^k$ and $\widetilde{t}(g) = t(g)$ for every element $g \in N$ is a surjective homomorphism. Statement   (\ref{it:tail extends}) follows.
 
To see Statement (\ref{it:kernel of tail}) we consider the group  $L_i(\mathcal{P})$ as defined above. We know that the projection of $L_i(\mathcal{P})$ on the coordinates $\prod_{j=1}^i \UP{j}$ is an isomorphism.  The tail $t(g)$ of every element $g = (g_k)_{k\in\NN} \in L_i(\mathcal{P})$ is equal to $g_i$.  
 It follows that the kernel of the projection of $L_i(\mathcal{P})$ to the coordinate $\UP{i}$ is equal to $\oplus_{j =1}^{i-1} \UP{j}$. Call this last group $R_i$ so that $R_i =  L_i(\mathcal{P}) \cap \ker t$. Now we have $N = \bigcup_{i=1}^\infty L_i(\mathcal{P})$ and so $\ker t = \U(\mathcal{P}) = \bigcup_i R_i$. It follows that $\U(\mathcal{P}) = \oplus_{i=1}^\infty \UP{i}$.
 \end{proof}

The structure of the group $\U(\mathcal{P})$ clearly   depends on the sequence $\mathcal{P}$, while  the   group $V = \UU{\infty}  \rtimes \left<\bar{\sigma} \right>$ is independent of  $\mathcal{P}$. As $\UU{\infty}$ and $\U(\mathcal{P})$ are both perfect groups, the derived subgroup of $G(\mathcal{P})$ is   $ \left[G(\mathcal{P}),G(\mathcal{P})\right] = N(\mathcal{P})$. The subgroup $\U(\mathcal{P})$ is generated by all the finite normal subgroups of $N(\mathcal{P})$ and is therefore characteristic.
 It follows that any isomorphism from $G(\mathcal{P})$ to $G(\mathcal{P}')$ would necessary take $\U(\mathcal{P})$ to $\U(\mathcal{P}')$. The groups  $\U(\mathcal{P})$ to $\U(\mathcal{P}')$  are isomorphic if and only if $\mathcal{P} = \mathcal{P}'$. This reproves B.H. Neumann's result saying that the groups $G(\mathcal{P})$ with different   $\mathcal{P}$'s are non-isomorphic, so that   there are uncountably many such $2$-generated groups.

The following commutative diagram depicts   three short exact sequences involving the B.H. Neumann group $G(\mathcal{P})$. The diagonal sequence is  the  abelianization map.
$$
 \begin{tikzcd}
1 \arrow[rd] & & & 1 \arrow[d] \\
& N(\mathcal{P}) \arrow[rd] \arrow[rr,"t"] & & \UU{\infty} \arrow[d] \arrow[r] & 1 \\
1 \arrow[r] &  \U(\mathcal{P}) \arrow[u,hook] \arrow[r] & G(\mathcal{P}) \arrow[rd, "\text{Ab}"] \arrow[r, "\widetilde{t}"] & V \arrow[d] \arrow[r] & 1 \\
& & & \left<\bar{\sigma}\right> \arrow[d] \arrow[rd] \\
& & & 1 & 1
 \end{tikzcd}
$$

\begin{remark}
\label{rmk:finite seq}
Consider the situation where $\mathcal{P}'$ is an infinite subsequence of $\mathcal{P}$.   The kernel of the natural map $G(\mathcal{P}) \to G(\mathcal{P}')$, obtained by deleting the coordinates of $\mathcal{P}$ which are not in $\mathcal{P}'$, is exactly 
$$\ker(G(\mathcal{P}) \to G(\mathcal{P}')) = \bigoplus_{\substack{i \in \NN\\ n_i \notin \mathcal{P}'}}  \UP{i} \le \U(\mathcal{P}).$$

The group  $G(\mathcal{P}')$ can be  formally defined   in exactly the same way also when the sequence $\mathcal{P}'$ is  finite. In that case the group $G(\mathcal{P}')$ is finite, and   unlike what happens when $\mathcal{P}'$ is infinite, the kernel of the natural map   $G(\mathcal{P}) \to G(\mathcal{P'})$ has finite index in $G(\mathcal{P})$. For a discussion of the finite index subgroups of $G(\mathcal{P})$ see \cite{lubotzky1996discrete}.
\end{remark}

\subsection*{On finite index subgroups}
We  give a pure group-theoretic property of the B.H. Neumann group $G = G(\mathcal{P})$ that will be  of crucial importance when we    classify all invariant random subgroups of $G$ in   \S\ref{sec:external case}.

%Recall the extended tail homomorphism $\widetilde{t} : G \to V = \UU{\infty} \rtimes \left<\bar{\sigma}\right>$ defined in Proposition \ref{prop:tail map extends}. In particular $\ker \widetilde{t} = \U(\mathcal{P})$.

\begin{prop}
\label{prop:large image implies finite index}
If $H$ is a subgroup of $G$ with $\left[V:\widetilde{t}(H)\right] < \infty$ then $\left[G:H\right] < \infty$.
\end{prop}

\begin{proof}
Let $H$ be a subgroup $G$ and assume that $\widetilde{t}(H)$ has finite index in $V$. The subgroup $\UU{\infty}$ of $V$ is simple and in particular has no finite index subgroups. Therefore $\widetilde{t}(H) = \UU{\infty} \rtimes \left<\bar{\sigma}^k\right>$ for some $k \in \NN$. 
The proof  will proceed in five steps.

\emph{Step 1. There exists some $i_H \in \NN$ such that for all $i > i_H$ the projection of the subgroup   $H$ to the coordinate $\UP{i}$  is surjective.}

Indeed, the image $\widetilde{t}(H)$ of $H$ contains the  subgroup $\UU{\infty}$. This means that the subgroup  $H$ contains  for each $ j \in \ZZ$ an element $h_j$ belonging to $N$ whose tail satisfies $\widetilde{t}(h_j) = t(h_j) = t(\tau^{\sigma^j})$. Moreover, as $\bar{\sigma}^k \in \widetilde{t}(H)$ for some $k \in \NN$, the subgroup $H$ contains an element $h$ satisfying $t(h) = \bar{\sigma}^k$.

As the tails of $h_j$ and $\tau^{\sigma^j}$ are the same,
for each $ j \in \setr{k-1}$  there exists an element $u_j \in \U(\mathcal{P})$ with $\tau^{\sigma^{j}} = u_j h_j$. Similarly there exists an element $u \in \U(\mathcal{P})$ so that   $h = u \sigma^k$.
Choose   $i_H \in \NN$ to be sufficiently large such that the $u_j$'s as well as $u$ are all contained in $\oplus_{i=1}^{i_H} \UP{i}$. 

 For every index $ i > i_H$ the projection of each element $h_j$ to the coordinate $\UP{i}$ is equal to $\tau_i^{\sigma_i^j}$, while the projection of the element $h$ is equal to $\sigma_i^k$. We claim that this projection contains the full generating set $C_i(\mathcal{P})$ for the group   $\UP{i}$ of  all alternating symmetries of $\setr{r_i}$. Indeed, this projection contains $\tau_i^{\sigma_i^l}$ for every $l \in \ZZ$, for if $l = ak+b$ for some $a,b \in \ZZ$ with $ 0 \le b < k$ then
$$\tau_i^{\sigma_i^l} = \left(\tau_i^{\sigma_i^b}\right)^{({\sigma_i^{k}})^a} \in C_i(\mathcal{P}).$$
We conclude that  for all $i > i_H$ the projection of the subgroup $H$ to the coordinate $\UP{i}$ is surjective.

\vspace{5pt}
\emph{Step 2. The kernel $H_0$ of the projection from $H$ to $\prod_{i=1}^{ i_H} \UP{i}$ has finite index in $H$. For all $i > i_H$ the projection of $H_0$  to the coordinate $\UP{i}$  is still surjective.}

For   $ i > i_H$, let $M_i$ be the kernel  of the projection from $H$ to the coordinate $\UP{i}$. Then $M_i$ it is a normal subgroup of $H$ with $H/M_i \cong \UP{i}$, while $H_0$ is a normal subgroup of $H$ such that all of the Jordan--Holder components of $H/H_0$ are simple groups of order less than $|\UP{i}|$. It follows that $H_0 M_i = H$ so that $H_0$ is mapped onto the coordinate $\UP{i}$ as required. 

We can therefore replace $H$ by $H_0$ and assume from now on that our subgroup $H$ is mapped trivially on the first $i_H$ coordinates and onto all the others.
\vspace{5pt}

\emph{Step 3. The subgroup   $K = H \cap \U(\mathcal{P})$ is  normal in $\U(\mathcal{P})$.}

 To this end, let $ x \in K$ and $y \in \U(\mathcal{P})$. Note that the first $i_H$-many coordinates of $x$ are trivial by our assumption on $H$. Since $x$ is also in $\U(\mathcal{P})$, there is some $i_x \in \NN$ so that for every $ i > i_x$ the coordinate of $x$ in $\UP{i}$ is trivial. 
 
 We claim that there exists an element $h \in H$ whose coordinates in $\UP{i}$ for all $i_H < i \le i_x$   are equal to those of the element $y$. This follows from Step 2. Indeed, the subgroup    $H$ is mapped onto each of the coordinates $\UP{i_H+1}, \ldots, \UP{i_x}$. Since these are different   simple groups $H$ is also mapped  onto their product, by Lemma \ref{lem:group theory}. Hence such an element $h$ does exist. 
 
Now, for this element $h$, we have that $x^y = x^h \in K^h$. Since $\U(\mathcal{P}) \nrm G(\mathcal{P})$ clearly   $K \nrm H$. Therefore  $x^y \in K$. We conclude that $K$ is normal in $\U(\mathcal{P})$.
\vspace{5pt}

\emph{Step 4. $K = \oplus_{i \in \NN} \U_{i}(\mathcal{P}')$ for some subsequence $\mathcal{P'}$ of $\mathcal{P}$.}

This follows from $H \cap \U(\mathcal{P}) \nrm \U(\mathcal{P})$ and $\U = \oplus_{i \in \NN} \UP{i}$.

\vspace{5pt}
\emph{Step 5.  The subsequence $\mathcal{P}'$ is co-finite in $\mathcal{P}$.}

Once we prove this, it follows that $H \cap \U(\mathcal{P})$ is of finite index in $\U(\mathcal{P})$ and altogether $H$ is of finite index in $G$.

Assume towards contradition that the complementary sequence $\mathcal{P}'' = \mathcal{P} \setminus \mathcal{P}'$ is infinite. Consider the map $G(\mathcal{P}) \to G(\mathcal{P}'')$ obtained by deleting all the coordinates $\UP{i}$ which are contained in the subgroup $H$, see Remark \ref{rmk:finite seq}. The image of $H$ in the group $G(\mathcal{P}'')$ via this map is isomorphic to 
$$Q =  H/(\oplus_{i \in \NN} \U_{i}(\mathcal{P}')) = H/K =  H / (H\cap \U(\mathcal{P})).$$ 
This means  that $Q$ is at the same time isomorphic to $\widetilde{t}(H)$, which is a finite index subgroup of $V$. This is impossible as  $G(\mathcal{P}'')$ is residually finite while $V$ is not.  Hence $\mathcal{P}''$ is finite. This completes the proof.
\end{proof}

\subsection*{Folner sets} We end Section  \ref{sec:Neumann groups} by constructing an explicit Folner sequence in the B.H. Neumann group $G(\mathcal{P})$.

\begin{prop}
\label{prop:Folner sets}
Denote $L_i = L_i(\mathcal{P})$. The subsets 
$$F_i = F_i(\mathcal{P}) =  \{\sigma^{j} l \: : \: j \in \setr{r_i-1}, l \in L_i \}$$
 form a Folner sequence in the group $G(\mathcal{P})$.
\end{prop}
 \begin{proof}
To see that the subsets $F_i$ form a Folner sequence it will suffice to verify that 
both quantities $|\sigma F_i \triangle F_i|$ and $|\tau F_i \triangle F_i|$ are $o(|F_i|)$. In the   case of the generator $\sigma$ we have that
$$ \frac{|\sigma F_i \triangle F_i|}{|F_i|} = \frac{2|L_i|}{(2r_i-1) | L_i|} = \frac{2}{2r_i-1}$$
as required. In the   case of the generator $\tau$, note that
$$ \tau F_i = \tau \{ \sigma^{j}  \: : \: j \in \setr{r_i-1} \} L_i =  \{\sigma^{j} \tau^{\sigma^j}   \: : \: j \in \setr{r_i-1}  \} L_i = F_i$$
as $\tau^{\sigma^j}  \in L_i$ for all   $j \in \setr{r_i-1}$. In particular $|\tau F_i \triangle F_i| = 0$ for all $i \in \NN$. This concludes the proof.
\end{proof}

\section{Invariant random subgroups and Weiss approximation}
\label{sec:Weiss approx}

We first 
recall  the notion of invariant random subgroups. We then discuss    Weiss approximations introduced in \cite{levit2019infinitely}. Those approximations   can be used in conjunction with the pointwise ergodic theorem for amenable groups to show that certain invariant random subgroups are co-sofic.
\subsection*{Invariant random subgroups}

We recall the basic theory of invariant random subgroups, mainly in order to set the notations. For proofs and more details --- see \cite{abert2014kesten} as well as \cite{becker2019stability, levit2019infinitely}.

%\subsection*{Space of subsets}

Let $G$ be a countable discrete group. 
Consider the \emph{power set}  
$ \Pow{G} = \{0,1\}^\Gamma$
equipped with the Tychonoff product topology. The space $\Pow{G}$ is compact.
The group $G$ acts on its power set  $\Pow{G}$ by homeomorphisms via conjugation. We denote this action by $c_g$, so that given an element $g \in G$ and a subset $A \subset G$ we have
$$c_g A = A^{g^{-1}} =   gAg^{-1}. $$

Consider the following subset of $\Pow{G}$
$$\Sub{G} = \{H \le G \: : \: \text{$H$ is a subgroup of $G$} \}.$$ 
The space $\Sub{G}$ is called the \emph{Chabauty space} of the group $G$. It is a closed subset of $\Pow{G}$ and hence is compact. It is clear that  $\Sub{G}$ is preserved by the conjugation action $c_g$ of $G$.
Let $\Subfi{G}$ denote the subset of $\Sub{G}$ consisting of all finite index subgroups.
%\subsection*{Spaces of probability measures} 

Let $\Probs{G}$ be the space of all Borel probability measures on the   set $\Sub{G}$. This space is  weak-$*$ compact   according to the Banach--Alaoglu theorem. The  conjugation action  of $G$ on its Chabauty space   $\Sub{G}$ gives rise to a corresponding push-forward action  of $G$ on the space $\Probs{G}$. We continue using the notation  $c_g$ for this push-forward action.

Denote
$$ \IRS{G} = \{ \mu \in \Probs{G} \: : \: \text{ $c_g \mu = \mu$ for all $g \in G$} \}.$$
Note that $\IRS{G}$ is a weak-$*$ closed and hence a compact subset of $\Probs{G}$. An element $\mu \in \IRS{G}$ is called an \emph{invariant random subgroup}.  	Let $\IRSfi{G}$ denote the subspace consisting of all $\mu \in \IRS{G}$ satisfying $\mu(\Subfi{G}) = 1$.

\begin{definition}
	\label{def:cosofic IRS}

	The invariant random subgroup $\mu \in \IRS{G}$ is \emph{co-sofic} if 
	$$\mu \in \overline{\IRSfi{G}}^{\text{weak-$*$}}.$$
\end{definition}

 The   space $\IRS{G}$ is a \emph{Choquet simplex} \cite[\S12]{phelps2001lectures}. This means that every invariant random subgroup $\mu \in \IRS{G}$ can be decomposed  as a convex combination of ergodic invariant random subgroups in a unique way. 

\subsection*{Chabauty spaces of subgroups and quotients}

Let $H$ be any subgroup of $G$. There is a continuous \emph{restriction map}  given by
$$ \cdot_{|H} : \Sub{G} \to \Sub{H}, \quad L \mapsto L_{|H} = L \cap H \quad \forall L \le G. $$
 The restriction map is $H$-equivariant for the conjugation action   of the subgroup $H$.
Pushing-forward via the restriction determines a map
$$ \cdot_{|H} : \IRS{G} \to \IRS{H}, \quad \mu \mapsto \mu_{|H} \quad \forall \mu \in \IRS{G}.$$

Let $Q$ be a quotient   of $G$ admitting   a surjective homomorphism $\pi : G \to Q$. There is a corresponding map $\pi : \Sub{G} \to \Sub{Q}$ of subgroups taking every subgroup $H \le G$ to its image $\pi(H)$ in $Q$.	The map $\pi : \Sub{G} \to \Sub{Q}$  is $G$-equivariant and Borel measurable. We obtain a push-forward map
$$ \pi_* : \IRS{G} \to \IRS{Q}, \quad \mu \mapsto \pi_*\mu \in \IRS{Q} \quad \forall \mu \in \IRS{G}.$$

\subsection*{Weiss approximable subgroups}
\label{sub:weiss approximable subgroups}

%We introduce a notion of Weiss approximable subgroups and relate this to co-soficity of invariant random subgroups. Weiss approximation is inspired by Benjamin Weiss' work \cite{weiss2001monotilable}.
%Our arguments   in this section crucially rely on the pointwise ergodic theorem for amenable groups due to Lindenstrauss \cite{lindenstrauss2001pointwise}. 

%\subsection*{Transversals and finite index subgroups}

Let $G$ be a   discrete  group and $H$ be a fixed subgroup of $G$.  
 	A \emph{transversal} for $H$ in $G$ is a subset $T_0$ consisting of one element from each   coset $tH$ of $H$ in $G$. A \emph{finite-to-one   transveral} $T$ for $H$ in $G$ is a disjoint union of finitely many   transversals for $H$ in $G$.
That is to say,    there is some $ k\in \NN$ so that $T$ consists of   exactly $k$ elements from each   coset of $H$.

%Let $H$ be a subgroup of the discrete group $G$.

\begin{definition}
	\label{def:cosofic subgroup}
	
	The subgroup $H$ is \emph{Weiss approximable} in $G$ if there are finite index subgroups $K_i$ of $G$ with finite-to-one transversals $F_i$ to $\mathrm{N}_G(K_i)$ such that
	$$ d_{\Probs{G}}\left(\frac{1}{|F_i|} \sum_{f \in F_i} \delta_{f K_i f^{-1}}, \frac{1}{|F_i|} \sum_{f \in F_i} \delta_{f H f^{-1}}\right) \xrightarrow{i \to\infty} 0 $$
	where $d_{\Probs{G}}$ is any compatible metric on the space $\Probs{G}$. We will say that the sequence $(K_i, F_i)$   is a \emph{Weiss approximation} for $H$ in $G$. 
\end{definition}

We now present an   explicit condition for a given sequence $(K_i, F_i)$   to constitute a Weiss approximation for the subgroup $H$.

\begin{prop}
	\label{prop:a numberical condition for a group to be co-sofic}
	Let $H$ be a subgroup of $G$. If  a given sequence  $K_i$ of finite index subgroups of $G$ with   finite-to-one transversals $F_i$ to $\mathrm{N}_G(K_i)$ satisfies
	$$ p_i(g) = \frac{|\{f \in F_i \: : \: g^f \in K_i \triangle H \}| }{|F_i|} \xrightarrow{i \to \infty} 0 $$
	for all elements $g \in G$ then $(K_i,F_i)$ is  a Weiss approximation for the subgroup $H$.
	%Then the subgroup $H$ is co-sofic.
\end{prop}
\begin{proof}
This is precisely  \cite[Proposition 3.7] {levit2019infinitely}.
\end{proof}

The following result will be our main tool in establishing that a given invariant random subgroup is co-sofic. 

\begin{theorem}
	\label{thm:cosofic subgroup implies cosofic IRS}
	Let $G$ be an amenable group and $\mu \in \IRS{G}$.  Let $F_i$ be a fixed  Folner sequence in the group $G$. If   $\mu$-almost every subgroup $H$ admits a Weiss approximation $(K_i, F_i)$ with some sequence   of finite index subgroups $K_i$ of $G$  then $\mu$ is co-sofic.
\end{theorem}

See \cite[Theorem  3.10]{levit2019infinitely}  for a detailed proof of our Theorem \ref{thm:cosofic subgroup implies cosofic IRS}. The proof crucially relies on the pointwise ergodic theorem for amenable groups due to Lindenstrauss \cite{lindenstrauss2001pointwise}. For the reader's convenience we reproduce the main ideas below.

\begin{proof}[Proof outline for Theorem \ref{thm:cosofic subgroup implies cosofic IRS}]
Making use of the ergodic decomposition, we may assume without loss of generality that the invariant random subgroup $\mu$ is ergodic. 

%See \cite[Corollary 2.6]{levit2019infinitely}  for a more detailed discussion.
 
Consider the probability measure preserving  action of the amenable group $G$ on the Borel space $(\Sub{G},\mu)$.  The pointwise ergodic theorem \cite{lindenstrauss2001pointwise} implies that for $\mu$-almost every subgroup $H \in \Sub{G}$ the sequence of probability measures $$\nu_i = \frac{1}{|F_{i}|} \sum_{f \in F_{i}} \delta_{c_f H}$$ converges to the invariant random subgroup $\mu$ as $i \to \infty$.  

By our assumption $\mu$-almost every subgroup   $H \in \Sub{G}$ is Weiss approximable  with respect to some sequence $(K_i, F_i)$. Therefore there exists a subgroup $H$  which is  Weiss approximable and at the same time the conclusion of the pointwise ergodic theorem applies to $H$. It follows from Definition \ref{def:cosofic subgroup} that the sequence of invariant random subgroups $$\mu_i = \frac{1}{|F_{i}|} \sum_{f \in F_{i}} \delta_{c_f K_i} \in \IRSfi{G}$$
satisfies $d_{\mathcal{G}}(\nu_i,\mu_i) \to 0$. Hence also $\mu_i$ converges to the invariant random subgroup $\mu$ as $i \to \infty$. We conclude that $\mu$ is co-sofic.
\end{proof}

While the ergodic theorem of \cite{lindenstrauss2001pointwise} requires the Folner sequence $F_i$ to be so called tempered, this can always be achieved by passing to a subsequence, so that this issue can be safely  ignored for our purposes.
  
\section{Invariant random subgroups of $V$}
\label{ref:vershik}

Consider the group  $\mathrm{Sym}_\text{fin}(\ZZ)$ of all finitely supported permutations of the set $\ZZ$ as well as its index two subgroup  $\UU{\infty} = \mathrm{Alt}_\text{fin}(\ZZ)$ consisting of  even permutations. 

Vershik \cite{vershik2012totally} had established a complete classification of all ergodic invariant random subgroups of the group $\mathrm{Sym}_\text{fin}(\ZZ)$. We recall this classification and apply it to study the invariant random subgroups of  $V = \UU{\infty} \rtimes \left<\bar{\sigma}\right>$. Our description follows closely that of Thomas' paper \cite{thomas2018characters}.

\subsection*{Random partitions}
Denote $\NN_0 = \NN \cup \{0\}$. Consider the  product  space $\NN_0^\ZZ$. To every point $\omega \in \NN_0^\ZZ$, we correspond the partition
$$ \ZZ = B^\omega_0 \amalg B^\omega_1 \amalg B^\omega_2  \amalg \cdots $$
where an integer $x \in \ZZ$ satisfies $x \in B^\omega_i$ if and only $\omega(x) = i$. Some members $B_i^\omega$ of this partition may be empty. Alternatively, we may think of the point $\omega$ as describing a coloring of the set $\ZZ$ using the colors $\NN_0$.

For $\omega \in \NN_0^\ZZ$, let $\mathrm{Sym}_\text{fin}(\omega)$ denote the   subgroup of $\mathrm{Sym}_\text{fin}(\ZZ)$ preserving the partition (or, alternatively, the coloring)  corresponding to $\omega$ and fixing $B_0^\omega$ pointwise. Namely, the color $ 0\in \NN_0$ is special in that $\mathrm{Sym}_\text{fin}(\omega)$ is required to fix every point $x \in \ZZ$ with $\omega(x) = 0$.   In other words
$$\mathrm{Sym}_\text{fin}(\omega) = \bigoplus_{i \in \NN} \mathrm{Sym}_\text{fin}(B_i^\omega).$$
In addition, there is a sign map
$$ \mathrm{sgn}_\omega :  \mathrm{Sym}_\text{fin}(\omega) \to \bigoplus_{i\in \NN} \ZZ/2\ZZ$$
  defined component-wise. Let 
$$\varepsilon : \bigoplus_{i\in \NN} \ZZ/2\ZZ \to \ZZ/2\ZZ$$
 be the map taking an element $\oplus_{i\in \NN} \ZZ/2\ZZ$ to the sum of all of  its   components. 
A permutation $g \in \mathrm{Sym}_\text{fin}(\omega)$ is even if and only if $\varepsilon \circ \mathrm{sgn}_\omega(g)=0$. The derived subgroup of $\mathrm{Sym}_\text{fin}(\omega)$ is the subgroup $\mathrm{Alt}_\text{fin}(\omega) =  \ker (\varepsilon  \circ \mathrm{sgn}_\omega)$.

We now describe a random model for such partitions.  Let $\alpha$ be a probability vector of the form
$$ \alpha = ( \alpha_0; \alpha_1, \alpha_2, \ldots ) $$
where $  \sum_{i\in\NN_0} \alpha_i = 1$, for all $ i \in \NN_0$ the $i$-th entry satisfies $0 \le \alpha_i \le 1$ and the sequence $(\alpha_i)_{i \in \NN}$  is monotone decreasing   in the sense that $\alpha_{i+1} \ge \alpha_i$ for all $i \in \NN$. It is \emph{not} required that $\alpha_0 \ge \alpha_1$. Consider the probability measure space $\Omega = (\NN_0^\ZZ, \alpha^\ZZ)$ where  $  \alpha^\ZZ$ is the Bernoulli probability measure. This means that the color of each point is chosen independently at random with respect to the probability measure    $\alpha$ at each coordinate. 
A \emph{random partition} of the set $\ZZ$ corresponding to the probability vector $\alpha$ is obtained  as described  above with respect to a   $\alpha^\ZZ$-random point $\omega \in \Omega$.

\begin{remark}
\label{rem:infinite}
Let $i \in    \NN_0$ be such that $\alpha_i > 0$. Then for  $\alpha^\ZZ$-almost every partition $\omega$ the member  $B_i^\omega$ is infinite, see e.g. \cite{vershik2012totally,thomas2018characters}.
\end{remark}

\begin{remark}
\label{rem:ergodicity}
The Bernoulli measure $\alpha^\ZZ$ is   $\mathrm{Sym}(\ZZ)$-invariant and the subgroup $\mathrm{Alt}_\text{fin}(\ZZ)$ is already acting ergodically. The proof of this fact is essentially the same as the well-known mixing argument for the ergodicity of the shift. See e.g. \cite[Theorem 1.30]{walters2000introduction}.
\end{remark}
 \subsection*{Vershik's classification theorem}

Roughly speaking, this theorem says that every invariant random subgroup of $\mathrm{Sym}_\text{fin}(\ZZ)$ arises from a random partition. More precisely (see also \cite{thomas2018characters}):

\begin{theorem}[\cite{vershik2012totally}]
\label{theorem:Vershik}
Let $\mu$ be an ergodic invariant random subgroup of $\mathrm{Sym}_\text{fin}(\ZZ)$. Then there is a probability vector $\alpha = (\alpha_0; \alpha_1, \alpha_2,\ldots)$ as above and  a subgroup $ S \le \oplus_{i \in \NN} \ZZ/2\ZZ$  such that $\mu = f^S_* \alpha^\ZZ$, where
$$  f^S : \Omega \to \Sub{\mathrm{Sym}_\text{fin}(\ZZ)}, \quad  f^S : \omega \mapsto \mathrm{sgn}_\omega^{-1}(S) \le \mathrm{Sym}_\text{fin}(\omega). $$
\end{theorem}

 Vershik's theorem deals with invariant random subgroups of $\mathrm{Sym}_\text{fin}(\ZZ)$ rather than of $\mathrm{Alt}_\text{fin}(\ZZ)$. Let us rapidly outline how to go from one group to the other.

\begin{cor} 
\label{cor:Vershik Alt}
Let $\mu$ be an ergodic invariant random subgroup of  $\mathrm{Alt}_\text{fin}(\ZZ)$. Then $\mu$   is obtained as in  Theorem \ref{theorem:Vershik} with the subgroup $ S \le \oplus_{i \in \NN} \ZZ/2\ZZ$   satisfying $\varepsilon(S) = 0$.
\end{cor}

In the following proof we rely on that fact that for any given group $G$ the space $\IRS{G}$ is a Choquet simplex. Therefore  every invariant random subgroup $\mu$ of $\mathrm{Alt}_\text{fin}(\ZZ)$ admits a \emph{unique} ergodic decomposition, and likewise for $\mathrm{Sym}_\text{fin}(\ZZ)$.

\begin{proof}[Proof of Corollary \ref{cor:Vershik Alt}]
 
Consider the element $\bar{\rho} \in \mathrm{Sym}_\text{fin}(\ZZ)$ given by $\bar{\rho} = (1,2)$, say. The alternating group $\mathrm{Alt}_\text{fin}(\ZZ)$ has index two in $\mathrm{Sym}_\text{fin}(\ZZ)$ and $\bar{\rho}$ represents the non-trivial coset.

It is easy to verify that the convex combination
$ \nu = \frac{1}{2}(\mu + \bar{\rho}_* \mu)$ is an invariant random subgroup of $\mathrm{Sym}_\text{fin}(\ZZ)$. If $\bar{\rho}_* \mu = \mu$ then     Vershik's Theorem \ref{theorem:Vershik} applied   to the invariant random subgroup $\nu = \mu$ directly implies  our Corollary.

We claim that $\bar{\rho}_*\mu \neq \mu$ is impossible. To see this, note that $\bar{\rho}_* \mu$ is  again an ergodic invariant random subgroup of $\mathrm{Alt}_\text{fin}(\ZZ)$.  So that the convex combination $\frac{1}{2}(\mu + \bar{\rho}_* \mu)$ is the $\mathrm{Alt}_\text{fin}(\ZZ)$-ergodic decomposition of $\nu$.  On the other hand, Theorem \ref{theorem:Vershik} combined with Remark 
\ref{rem:ergodicity} implies that every ergodic $\mathrm{Sym}_\text{fin}(\ZZ)$-invariant random subgroup is $\mathrm{Alt}_\text{fin}(\ZZ)$-ergodic. Therefore the $\mathrm{Sym}_\text{fin}(\ZZ)$-ergodic decomposition of $\nu$ coincides with its ergodic decomposition regarded as an invariant random subgroup of $\mathrm{Alt}_\text{fin}(\ZZ)$. This eliminates the possibility that $\bar{\rho}_*\mu \neq \mu$, as required.
\end{proof}

Recall that $\bar{\sigma} \in \mathrm{Sym}(\ZZ)$ is the infinitely supported permutation acting via $\bar{\sigma}(x) = x+1$ for all $x \in \ZZ$.  Moreover recall the notation $\UU{\infty} = \mathrm{Alt}_\text{fin}(\ZZ)$.

\begin{cor}
\label{cor:structure of normalized IRS of Uinft}
Let $\mu$ be an ergodic invariant random subgroup of $\UU{\infty}$. If $\mu$-almost every subgroup is normalized by $\bar{\sigma}^k u$ for some fixed $k \in \NN$ and some element $u \in \UU{\infty}$ depending on the subgroup,  then either $\mu = \delta_{\{e\}}$ or $\mu = \delta_{\UU{\infty}}$.
\end{cor}
\begin{proof}
Apply Corollary \ref{cor:Vershik Alt} to classify the invariant random subgroup $\mu$. Let $\alpha$ be the resulting probability vector and $  l(\alpha) \in \NN \cup \{\infty\}$ be the number of the non-zero entries in  $\alpha$. We claim that $l(\alpha) = 1$. The conclusion follows immediately from this, for the only two possible probability vectors with $l(\alpha) = 1$ are either $\alpha_0 = (1;0,0,\ldots)$ or $\alpha_1 = (0;1,0,\ldots)$. The first one gives rise to the invariant random subgroup $\delta_{\{e\}}$ and the second one to $\delta_{\UU{\infty}}$.

Assume towards contradiction that $l(\alpha) \ge 2$. This means that $\mu$-almost  every subgroup preserves a   random partition $\omega$ admitting at least two infinite members, see Remark \ref{rem:infinite}. Such a subgroup is normalized by some $\bar{\sigma}^k u$ if and only if $\bar{\sigma}^k u$ leaves  the block $B_0^\omega$ invariant and permutes the blocks $B_1^\omega, B_2^\omega,\ldots$ among themselves. The probability that $\bar{\sigma}^k$ acts in such a   structured way with respect to a random partition $\omega$ is zero.
Since $\bar{\sigma}^k u$ differs from $\bar{\sigma}^k$ on at most a finite number of points, the probability with respect to $\mu$ of this happening is also zero. We arrive at a contradiction.
\end{proof}

\subsection*{Invariant random subgroups of $V$}
Recall that   $V = \UU{\infty} \rtimes \left<\bar{\sigma}\right>$.% where $\bar{\sigma} \in \mathrm{Sym}(\ZZ)$ is the permutation $\bar{\sigma}(x) = x+1$ for all $x \in \ZZ$.
\begin{cor}
\label{cor:Vershik IRS}
Let $\mu$ be an ergodic invariant random subgroup of $V$. Then exactly one of the  following two statements is true:
\begin{enumerate}
\item $\mu$-almost every subgroup satisfies $H \le \UU{\infty}$, or
\item $ \mu $ is the atomic probability measure supported on the subgroup $\UU{\infty}\rtimes \left<\bar{\sigma}^k\right>$ for some $k \in \NN$.
\end{enumerate}
\end{cor}
\begin{proof}
It is clear that the two possibilities are mutually exclusive. Assume that Condition (1) fails. As $\mu$ is ergodic, this means that $\mu$-almost every subgroup contains an element of the form $\bar{\sigma}^k u$ for a fixed $k \in \NN$ and some $u \in \UU{\infty}$ depending on that subgroup. The restriction of $\mu$ to $\UU{\infty}$ is an invariant random subgroup of $\UU{\infty}$. It follows from Corollary  \ref{cor:structure of normalized IRS of Uinft} that this restriction is   equal to a convex combination  $s \delta_{\{e\}} + (1-s)  \delta_{\UU{\infty}}$ for some $s \in \left[0,1\right]$. Write $\mu = s \mu' + (1-s) \delta_{\UU{\infty}\rtimes \left<\bar{\sigma}^k\right>}$ where $\mu'$-almost every subgroup has trivial intersection with $\UU{\infty}$. This means that $\mu'$-almost every subgroup is cyclic and generated by an element of the form $\bar{\sigma}^k u$. Since $V$ admits only countably many such cyclic subgroups and each has an infinite conjugacy class, we conclude that $s = 0$.  Condition (2) follows.
\end{proof}

\section{Invariant random subgroups supported on the derived subgroup}
\label{sec:internal case}

Let $\mathcal{P} = (n_1, n_2, \ldots)$ be a monotone increasing sequence of integers with $n_1 \ge 5$. Consider the corresponding B.H. Neumann group $G = G(\mathcal{P})$. In this section we construct a Weiss approximation to every subgroup of $G$ contained in its derived subgroup $N = N(\mathcal{P}) = \left[G(\mathcal{P}), G(\mathcal{P})\right]$.

Write $n_i = 2r_i+1$. Recall the finite subgroups $L_i(\mathcal{P})$ defined in \S\ref{sec:Neumann groups} for every $i \in \NN$.  In particular $N$ is the ascending union of the $L_i (\mathcal{P})$'s. Denote $L_i = L_i(\mathcal{P})$.

Consider the   subgroups $G_i$ given for each $i \in \NN$  by
$$ G_i = \ker \left(G(\mathcal{P}) \to \prod_{j=1}^i \UP{j}\right).$$
Since the projection of each $L_i $ to the product $ \prod_{j=1}^i \UP{j}$ is an isomorphism it follows that $L_i $ is a complement to $G_i$ in the group $G$, namely for each $i \in \NN$
$$  L_i  \cap G_i = \{e\} \quad \text{and} \quad  L_i G_i = G.$$ 
Finally, consider the abelianization map $G \to \left<\bar{\sigma}\right> \cong \ZZ$ and take 
$$ D_i = \ker \left( G \to \ZZ/(2r_i - 1)\ZZ \right). $$
Clearly both $G_i$ and $D_i$ are finite index normal subgroups of $G$ for every $i \in \NN$.

\begin{prop}
\label{prop:adapted}
Fix an element  $g \in N$. Then
$$ \frac{| \{j \in \setr{r_i}	\: : \: g^{\sigma^j} \in L_i \}| }{| \setr{r_i} | } \xrightarrow{i \to \infty} 1. $$
\end{prop}
\begin{proof}
Recall that $i = i(g) \in \NN$ is the smallest index so that $g \in L_i(\mathcal{P})$. We will prove the stronger statement that 
$$ \frac{| \{j \in \setr{r_i}	\: : \: L_{i(g)}^{\sigma^j} \le L_i \}| }{| \setr{r_i} | } \xrightarrow{i \to \infty} 1. $$
It follows from the definition of the groups $L_i$ that the condition $L_{i(g)}^{\sigma^j} \le L_i$ will hold provided that $j + \setr{i(g)} \subset \setr{r_i}$. This last condition is satisfied whenever $j  \in \setr{r_i-i(g)}$. As the index   $i(g)$ is being kept fixed, we have
$$ \frac{| \setr{r_i-i(g)}| }{| \setr{r_i} |} = \frac{2(r_i-i(g)) + 1}{2r_i+1}   \xrightarrow{i \to \infty} 1$$
as requried.
\end{proof}
The above Proposition \ref{prop:adapted} is to be compared with the notion of \emph{adapted subsets}  in the setting of metabelian groups  \cite[Definition 7.2 and Lemma 10.3]{levit2019infinitely}. 

We now construct the desired Weiss approximations with respect to the Folner sequence $F_i = F_i(\mathcal{P}) = \{\sigma^{j}  \: : \: j \in \setr{r_i-1} \} L_i $ studied in Proposition \ref{prop:Folner sets}.  

\begin{prop}
\label{prop:Weiss approx}
Let $H \le N$ be any subgroup. Then $H$ has a   Weiss approximation $(K_i, F_i)$ for some sequence $K_i \le G$ of finite index subgroups.
\end{prop}
\begin{proof}
Fix   $i \in \NN$ and consider the subgroup $K_i$ of $G$   given by
$$ K_i = \left(H\cap L_i \right)\left(G_i \cap D_i\right).$$
Since   $G_i \cap D_i$ is a finite index normal subgroup of $G$ it follows that   $K_i$ has finite index in $G$ as well. Moreover, since $\{\sigma^k \: :\: k \in \setr{r_i - 1} \}$ is a transversal to the normal subgroup $D_i$ and $L_i$ is a transversal to the normal subgroup $G_i$, the subset $F_i$ is a finite-to-one transversal to $G_i \cap D_i$ (actually,   it is a transversal since $G_i D_i = G$, but this fact is not needed). We deduce that   $F_i$ is a finite-to-one transversal to the subgroup $K_i$.

We claim that for all $i \in \NN$, the condition  $L_i \subset G \setminus (H \triangle K_i)$ holds. This condition  is equivalent to $$L_i \cap H = L_i \cap K_i.$$ The inclusion  $L_i \cap H \subset L_i \cap K_i$ is clear. For the converse, consider some $g \in L_i \cap K_i$ so that $g = h d$ for some pair of elements $h \in H \cap L_i$ and $d \in G_i \cap D_i$. As both $g$ and $h$ are contained in $L_i$ it follows that $ d \in L_i$ as well. However $L_i \cap G_i = \{e\}$ and $d \in G_i$ so that $d$ must be trivial. Hence  $g = h \in L_i \cap H$. 

We would like to apply  Proposition	\ref{prop:a numberical condition for a group to be co-sofic} and deduce that   $(K_i, F_i)$ is a Weiss approximation. In other words,   we need to verify that every element $g \in G $ satisfies
$$ p_i(g) = \frac{|\{f \in F_i \: : \: g^f \in K_i \triangle H \}| }{|F_i|} \xrightarrow{i \to \infty} 0. $$
To this end, fix an element $g \in G $.   First consider  the case where  $g \notin N$. Therefore   $g^f \notin H$ for all $f \in G$. Similarly, provided   $i \in \NN $ is sufficiently large $g^f \notin D_i$ which implies $g^f \notin K_i$ for all $f \in G$. In particular $p_i(g) = 0$ for all $i \in \NN$ sufficiently large.

It remains to consider the more interesting case   where $g \in N$. Since by the above $L_i$ belongs to   $G \setminus (H \triangle K_i)$,   it will suffice to show that
$$ q_i(g) = \frac{|\{f \in F_i \: : \: g^f \in L_i \}| }{|F_i |} \xrightarrow{i \to \infty} 1. $$
Every element $f \in F_i$ can be written as  $f = \sigma^j l $ where $j \in \setr{r_i - 1}$ and $l \in L_i$. Observe that 
$$ g^{f} \in L_i \Leftrightarrow g^{\sigma^j l} \in L_i \Leftrightarrow g^{\sigma^j  } \in L_i.$$
The element $g$ being kept fixed, the probability that $g^{\sigma^j  } \in L_i$ as $j$ ranges over the set $\setr{r_i - 1}$ tends to one as $i$ goes to infinity, see Proposition \ref{prop:adapted}. This completes the proof.
\end{proof}

Returning to the remark made in the final paragraph of \S\ref{sec:intro}, note that a special case    of Proposition \ref{prop:Weiss approx} constructs a Weiss approximation $(K_i, F_i)$ to the  subgroup $U = \U(\mathcal{P})$. In this case, one can check that the proof gives $$K_i =\ker(G(\mathcal{P}) \to \UP{i}).$$ It is interesting to note that for all $i \in \NN$ sufficiently large, the finite index subgroups $K_i$ do not contain the subgroup $U$. Indeed, the only finite index subgroups containing $U$ are the $D_i$'s, and they are clearly do not form a Weiss approximation to the subgroup $U$.

\section{Proof of the main theorem}
\label{sec:external case}

Let $\mu \in \IRS{G(\mathcal{P})}$ be any invariant random subgroup. We   now  show that $\mu$ is co-sofic, thereby completing the proof of   Theorem \ref{thm:main result IRS}. Our main result Theorem \ref{thm:main theorem } follows,   relying on the criterion   \cite{becker2019stability} reproduced here as  Theorem \ref{thm:BLT}.

We may assume  without loss of generality that $\mu$ is ergodic, see  \cite[Corollary 2.6]{levit2019infinitely} for clarification.

First consider     the case   where $\mu$-almost every subgroup $H$ is contained in $N$.  Therefore $\mu$-almost every subgroup admits a Weiss approximation $(K_i, F_i)$ with respect to our fixed Folner sequence $F_i = F_i(\mathcal{P})$ of   finite-to-one transversals as established in Proposition \ref{prop:Weiss approx}.  The fact that $\mu$ is co-sofic   follows from Theorem \ref{thm:cosofic subgroup implies cosofic IRS} and relying on  the Lindenstrauss pointwise ergodic theorem for amenable groups.

The second case to consider is where the projection of $\mu$ to the abelianization of $G$ is a non-trivial ergodic invariant random subgroup. Namely, the image of $\mu$-almost every subgroup $H$ in the  $G/G'$ is equal to $\left<\bar{\sigma}\right>^k$ for some $k \in \NN$. We may therefore consider the resulting push-forward invariant random subgroup on $V = \UU{\infty} \rtimes \left<\bar{\sigma}\right>$. Vershik's theorem and in particular its Corollary \ref{cor:Vershik IRS} implies that for $\mu$-almost every subgroup $H$ its image in $V$ has finite index. But according to the algebraic Proposition \ref{prop:large image implies finite index}, this means  that $\mu$-almost every subgroup has finite index in $G$ to begin with. In other words   $\mu \in \IRSfi{G}$, which is already a much stronger conclusion than $\mu$   being co-sofic. The theorem is now proved.

%\noindent\rule{\textwidth}{1pt}

\bibliography{Neumann}
\bibliographystyle{alpha}

\end{document}